\documentclass{amsart}

\usepackage{latexsym,enumerate}
\usepackage{amsmath,amsthm,amsopn,amstext,amscd,amsfonts,amssymb}
\usepackage[ansinew]{inputenc}
\usepackage{verbatim}
\usepackage{pstricks,pst-node}
\usepackage{wasysym}
\usepackage[all]{xy}
\usepackage{epsfig} 
\usepackage{graphicx,psfrag} 
\usepackage{subfigure}
\usepackage{url}

\theoremstyle{theorem}
\newtheorem{theorem}{Theorem}
 
\theoremstyle{definition}

\newtheorem*{theorem1}{Theorem 1}

\newtheorem{lemma}[theorem]{Lemma}
\newtheorem{proposition}[theorem]{Proposition}

\DeclareMathOperator{\Deg}{Deg}
\renewcommand{\d}{\delta}

\newcommand{\D}{\Delta}

\begin{document}


\title{Distinguishing wheel graphs by the alliance polynomial}


\author[W. Carballosa]{Walter Carballosa}
\address{Department of Mathematics and Statistics, Florida International University, 11200 S.W. 8th Street, Miami FL 33199, USA}

\author[O. Rosario]{Omar Rosario}
\address{Faculty of Mathematics, Autonomous University of Guerrero, Carlos E. Adame 54, La Garita, Acapulco 39650, Mexico}

\author[J.M. Sigarreta]{Jos\'e M. Sigarreta}
\address{Institute of Physics "Luis Rivera Terrazas", Benem\'erita Autonomous University of Puebla, Av. San Claudio Esq. 18 Sur, Edificio IF1, Ciudad Universitaria, Col. San Manuel, Puebla, Puebla,  M\'exico}

\author[Y. Torres-Nu\~nez]{Yadira Torres-Nu\~nez}
\address{Department of Mathematics, Miami Dade College, Wolfson Campus, 300 NE Second Ave., Miami, FL 33132, USA}

\maketitle{}

\begin{abstract}
Distinctive power of the alliance polynomial has been studied in previous works, for instance, it has been proved that the empty, path, cycle, complete, complete without one edge and star graphs are characterized by its alliance polynomial. Moreover, it has been proved that the family of alliance polynomial of regular graphs with small degree is a very special one, since it does not contain alliance polynomial of graphs other than regular graph with the same degree. In this work we prove that the alliance polynomial also determines the wheel graphs. 





\end{abstract}


\section{Introduction.}

Several graph parameters have been used to define a graph polynomial, for instance,
the parameters associated to matching sets \cite{F}, independent sets \cite{HL}, domination sets \cite{ALL}, differential number \cite{BCLS}, chromatic numbers \cite{T} and many others.
In \cite{CRST}, authors introduce the alliance polynomial of a graph using the exact index of alliance. 
In \cite{CRST3}, they study the alliance polynomials of regular graphs.

We begin by stating the used terminology.  Throughout
this paper, $G=(V,E)$ denotes a simple graph of order $|V|=n$ and size $|E|=m$.
We denote two adjacent vertices $u$ and $v$ by
$u\sim v$. For a nonempty set $X\subseteq V$, and a vertex $v\in V$,
 $N_X(v)$ denotes the set of neighbors  $v$ has in $X$:
$N_X(v):=\{u\in X: u\sim v\},$ and the degree of $v$ in $ X$ will be
denoted by $\delta_{X}(v)=|N_{X}(v)|$.
We denote by $\delta$ and $\Delta$ the minimum and maximum degree of $G$, respectively.
In particular, if $G$ is a regular graph we denote by $\Delta$ the degree of $G$, in fact, we write $G$ is $\Delta$-regular.
The subgraph induced by
$S\subset V$ will be denoted by  $\langle S\rangle $ and the
complement of the set $S\in V$ will be denoted by $\bar{S}$.

A nonempty set $S\subseteq V$ is a \emph{defensive  $k$-alliance} in
$G=(V,E)$,  $k\in [-\Delta,\Delta]\cap \mathbb{Z}$, if $\delta _S(v)\ge \delta_{\bar{S}}(v)+k$ for
every $ v\in S$.
We consider the value of $k$ in the set of integers $\mathcal{K}:= [-\Delta,\Delta]\cap \mathbb{Z}$.
In some graph $G$, there are some values of  $k\in \mathcal{K}$, such that do not exist defensive
$k$-alliances in $G$. For instance, in the star graph $S_n$ with $n$ vertices, does no exist defensive $k$-alliances for $k\ge 2$; besides, $V(G)$ is a defensive $\delta$-alliance in $G$.
Notice that for each $S\subset V(G)$ there exists $k\in\mathcal{K}$ such that $S$ is a defensive $k$-alliance in $G$.
We denote by $k_{S}:=\displaystyle\max_{} \{k\in\mathcal{K} \,:\,  S \text{ is a defensive $k$-alliance}\}$.
We say that  $k_{S}$ is the \emph{exact index of alliance} $S$, or also, $S$ is an \emph{exact defensive $k_{S}$-alliance} in $G$.

Let $G$ be a graph.
We define the \emph{alliance polynomial} of a graph $G$ as
\begin{equation}\label{eq:Poly1}
    A(G;x)= \displaystyle\sum_{k\in \mathcal{K}} A_{k}(G) x^{n+k},
\end{equation}
 with  $A_{k}(G)$  the number of connected exact defensive $k$-alliances in $G$.

We say that a graph $G$ is characterized by a graph polynomial $f$ if for every graph $G^\prime$ such that $f(G^\prime) =f(G)$ we have that $G^\prime$ is isomorphic to $G$. The class of graphs $\mathcal{F}$ is characterized by a graph polynomial $f$ 
if every graph $G\in\mathcal{F}$ is characterized by $f$.
This notion has been studied for the chromatic polynomial, the Tutte polynomial and the matching polynomial in \cite{MN,N}. It is shown, \emph{e.g.},that several well-known families of graphs are determined by their Tutte polynomial, among them the class of wheels, squares of cycles, complete multipartite graphs, ladders, M\"{o}bius ladders, and hypercubes.

The distinctive power of the alliance polynomial has been studied in previous works, for instance, in \cite{CRST} authors proved that the empty, path, cycle, complete, complete without one edge and star graphs are characterized by their alliance polynomials. Moreover, in \cite{CRST3} was proved that the family of alliance polynomial of $\Delta$-regular graphs with small degree (\emph{i.e.}, $\Delta\le5$) is also characterized by the alliance polynomial. In this work we prove that the alliance polynomial also characterizes the wheel graphs. 

The main result in this paper is Theorem \ref{t:Wn}. In our attempt to prove it we obtain a close formula for the alliance polynomial of the wheel graphs using combinatorial arguments, see Proposition \ref{p:Wn}. The proof of Theorem \ref{t:Wn} is given in the next section.

\begin{theorem}\label{t:Wn}
Wheel graphs are characterized by the alliance polynomial. 
\end{theorem}

\section{Proof of the main result}\label{wheel}

Let $G_1=(V_1,E_1)$ and $G_2=(V_2,E_2)$ two graphs with $V_1\cap V_2=\varnothing$. The \emph{graph join} $G_1+G_2$ of $G_1$ and $G_2$ has $V(G_1+G_2)=V_1 \cup V_2$ and two different vertices $u$ and $v$ of $G_1+G_2$ are adjacent if $u\in V_1$ and $v\in V_2$, or $[u,v]\in E_1$ or $[u,v]\in E_2$. 
The following result that will be useful appears in \cite[Theorem 2.12]{CRST}

\begin{theorem}\label{t:join}
  Let $G_1,G_2$ be two graphs with order $n_1$ and $n_2$, respectively. Then
  \[
  A(G_1+ G_2;x)= A(G_1;x) + A(G_2;x)+ \widetilde{A}(G_1,G_2;x),
  \]
  where $\widetilde{A}(G_1,G_2;x)$ is a polynomials with $\widetilde{A}(G_1,G_2;1)=(2^{n_1}-1)(2^{n_2}-1)$ and $Deg\big(\widetilde{A}(G_1,G_2;x)\big)=Deg\big(A(G_1\cup G_2;x)\big)$.
\end{theorem}

Polynomial $\widetilde{A}(G_1,G_2;x)$ is associated to defensive alliances in $G_1+ G_2$ with no null part in both original graphs, i.e., associated to defensive alliances $S\subset V(G_1+ G_2)$ with $S\cap V(G_1) \neq \emptyset$ and $S\cap V(G_2) \neq \emptyset$.
Since wheel graph $W_{n}$ is an isomorphic graph to $E_1+ C_{n-1}$ for every $n\ge4$, Theorem \ref{t:join} has the following consequences.

\begin{proposition}\label{p:Wn}
   Let $W_n$ be a wheel graph with order $n\geq 4$. Then
   \begin{equation}\label{eq:Wn}
     \begin{aligned}
      A(W_n;x) =
      &  A(E_1;x) + A(C_{n-1};x) + \sum_{k=2}^{\lfloor n/2\rfloor} {n-1 \choose k-1} \ x^{2k-1} + \sum_{k=\lfloor n/2 \rfloor + 1}^{n-1} a_{n,k} \ x^{n-1}  \\
      & +  \xi_n \, b_{n,(n+1)/2} \ x^n  + \sum_{k= \lceil n/2\rceil + 1}^{n-1} b_{n,k} \ x^{n+1} + x^{n+3}.
     \end{aligned}
   \end{equation}
   where, $\xi_n = 0$ if $n$ is even and $\xi_n = 1$ if $n$ is odd,
   \[
   b_{n,k} = (n-1)\displaystyle\sum_{r=1}^{\lfloor\frac{k-1}2\rfloor}  \frac{{n-k-1 \choose r-1} {k-1-r \choose r}}{k-1-r}
   \quad\quad {\text and } \quad\quad a_{n,k} = {n-1 \choose k-1} - b_{n,k}.\]
\end{proposition}

\begin{proof}
   We remark, from the geometry point of view, that wheel graph $W_n$ is union of one star graph $S_n$ and a cycle $C_{n-1}$ obtained by identifying the pendants vertices of $S_n$ with the vertices of $C_{n-1}$. We colloquially refer to $W_n$ as a graph separated in two parts an $n$-star $S_n$ and an ($n-1$)-cycle $C_{n-1}$.

   By Theorem \ref{t:join} we have to compute $\widetilde{A}(E_1,C_{n-1};x)$.
   Fix $n \geq 4$. In order to compute $\widetilde{A}(E_1,C_{n-1};x)$, let us consider $S\subset V(W_n)$ containing the central vertex and $|S|\ge2$.
   Hence, we  separately analyze subsets $S$ with different cardinality.

   \begin{description}
     \item[$2\leq |S|=k \leq n-1$] {Since $S$ contains the center vertex there are ${n-1 \choose k-1}$ ways to choose $S\cap V(C_{n-1})$, but we can obtain defensive alliances with two different exact indices as below

         \begin{itemize}
            \item {There is an isolated  vertex in $S\cap V(C_{n-1})$. For each $S$ we have an exact defensive $\min\{k-1 - (n-1-k+1), 1-2\}$-alliance since the others vertices in $C_{n-1}$ do not contribute to the exact index of $S$.
                Denote the number of those $S$ by $a_{n,k}$. Therefore, does appear the addend $a_{n,k} x^{n + \min\{2k-1-n, -1\}}$.
                }
            \item {None of the $k-1$ vertices in $C_{n-1}$ is isolated. So, all $S$ are exact defensive $\min\{k-1 - (n-1-k+1), 2-1\}$-alliances.
                Denote the number of those $S$ by $b_{n,k}$. Therefore, does appear the addend $b_{n,k} x^{n + \min\{2k-1-n, 1\}}$.      
                }
         \end{itemize}
         }

     \item[$|S|= n$] {In this case $S=V(W_n)$ and induces the wheel graph $W_n$. So, we have an exact defensive $3$-alliance, since $\d_{W_n} = 3$ and $W_n$ is connected. So, the addend $x^{n + 3}$ appears.
         }
   \end{description}

   Now we compute $b_{n,k}$ by solving the following combinatorial problem

   \emph{How many cyclic binary strings of length $n-1$ have exactly $k-1$ numbers $1$'s such that does not appear the substring $010$?}\footnote{Note that cyclic binary strings identify rotations, \emph{i.e.}, string $100$ is different to $010$ in spite of a shifting (rotation).}

   Now enumerate the vertices of $C_{n-1}$ from $1$ to $n-1$ such that consecutive labels are adjacent. So, it is equivalent to count configurations of $r$ blocks of several $1$s separated by at least one zero for $1<r<\lceil n\rceil/2$. Hence, first choose configurations of $r$ blocks of $1$s interlaced in a chain of $n-k$ zeros, and secondly select the length of the $r$ blocks of $1$s, \emph{i.e.}, $r$ addends with sum $k-1$ where each addend is greater than or equal two. Analyze separately the following cases
   
   \begin{itemize}
     \item[I]   {
                \begin{figure}[h]
                \centering
                \scalebox{1.5}
                {\begin{pspicture}(-0.2,-0.2)(0.2,-0.2)
                \uput[0](0,0){$0\_\_0\_\_0\_\_\ldots 0\_\_0\_\_$}
                \end{pspicture}}
                \end{figure}

                \vspace{-.3cm}

                If the first one is  $0$ (see above), then there are ${n-k \choose r} {k-2-r \choose r-1}$ different ways.
                }   
     \item[II]  {
                \begin{figure}[h]
                \centering
                \scalebox{1.5}
                {\begin{pspicture}(-0.2,-0.1)(0.2,-0.2)
                \uput[0](0,0){$10\_\_0\_\_0\_\_\ldots 0\_\_0.1$}
                \end{pspicture}}
                \end{figure}

                \vspace{-.3cm}

                If the first one is $1$ and second is $0$ (so the last one is $1$, too; see above), then there are ${n-k-1 \choose r-1} {k-2-r \choose r-1}$ different ways.
                }               
     \item[III]  {
                \begin{figure}[h]
                \centering
                \scalebox{1.5}
                {\begin{pspicture}(-0.2,-0.1)(0.2,-0.2)
                \uput[0](0,0){$11.0\_\_0\_\_0\_\_\ldots 0\_\_0$}
                \end{pspicture}}
                \end{figure}

                \vspace{-.3cm}

                If the firsts two are $1$s and the last one is $0$ (see above), then there are ${n-k-1 \choose r-1} {k-2-r \choose r-1}$ different ways.
                }  
     \item[IV]  {
                \begin{figure}[h]
                \centering
                \scalebox{1.5}
                {\begin{pspicture}(-0.2,-0.1)(0.2,-0.2)
               \uput[0](0,0){$11.0\_\_0\_\_0\_\_\ldots 0\_\_0.1$}
                \end{pspicture}}
                \end{figure}

                \vspace{-.3cm}

                If the firsts two and the last one are $1$s (see above), then there are ${n-k-1 \choose r-1} {k-2-r \choose r}$ different ways.
                }  
    \end{itemize}



%
\end{proof}

The following results will be useful which appear in \cite[Proposition 3.4]{CRST}.

\begin{proposition}\label{p:Disjoint_Cn}
   Let $C_n$ be a cycle graph with order $n\geq 3$, then
         \begin{equation}\label{eq:Cn}
           A(C_n;x) = n \, x^{n-2} + n(n-2) \, x^n + x^{n+2}.
         \end{equation}
\end{proposition}

Let $A(G;x)$ be the alliance polynomial of a graph $G$, we denote by $\Deg(A(G;x))$ and $\Deg_{\min}(A(G;x))$ the maximum degree and minimum degree of its terms, respectively.
We recall some previous results for general graphs will be useful which appear in \cite[Theorems 2.4 and 2.5]{CRST}.

\begin{lemma}\label{p:AlliPoly}
   Let $G$ be a graph. Then, $A(G;x)$ satisfies the following conditions

   \begin{enumerate}[i)]
     \item {$A_{-\D}(G)$ and $A_{-\D+1}(G)$ are the number of vertices in $G$ with degree $\D$ and $\D-1$, respectively.}
     
     \item {$A(G; 1) < 2^n$ is the number of connected induced subgraphs in $G$.}

     \item {$A_{\D}(G)$ is equal to the number of connected components in $G$ which are $\D$-regular.}

     \item {$n + \d \leq \Deg(A(G;x)) \leq n + \D$.}
\end{enumerate}
\end{lemma}

The following result allows to characterize the wheel graphs by its alliance polynomials. 

\begin{theorem1}
Wheel graphs are characterized by the alliance polynomial. 
\end{theorem1}

\begin{proof}
Let $W_n$ be a wheel graph with order $n\ge4$ and $G$ a graph such that $A(G;x)=A(W_n;x)$. 
   Suppose that $G$ has order $n_1$ and degree sequence $\{\delta_{1}, \delta_{2}, \ldots,\delta_{n_1}\}$ (ordered as follows $\delta_{1}\geq \delta_{2}\geq \cdots \geq\delta_{n_1}$).
   Since $\Deg_{min}(A(G;x)) = 1$, there is $w\in V(G)$ such that for every $v\in V(G) \setminus\{w\}$ we have $v \sim w$; we say that $w$ is a central vertex in $G$ (notice that by Lemma \ref{p:AlliPoly} item 1 it is unique). Therefore, $G$ is a connected graph, $\d_{n_1} \geq 1$ and $G$ contains an isomorphic subgraph of $S_{n_1}$, we denote this subgraph by $G_S$.

   Since $\Deg(A(G;x)) = n + 3$, we have
   \begin{equation}\label{eq:Wn1}
   n_1 + \d_{n_1} \leq n+3, \quad \Longrightarrow \quad n_1 \leq n + 2. 
   \end{equation}

   By \eqref{eq:Wn}, we have
   \[
   \begin{aligned}
   A(G;1) & =  1 + (n-1) + (n-1)(n-3) + 1 + \sum_{k=2}^{n} {n-1 \choose k-1} \\
   & = (n-1)(n-2) + 1 + 2^{n-1} > 2^{n-1}.
   \end{aligned}
   \]

   By Lemma \ref{p:AlliPoly} item 2 we have $n_1 \geq n$.
   Then, $n \leq n_1  \leq n + 2$.

   If $n_1 = n + 2$, then $\d_{n_1} = 1$ by Lemma \ref{p:AlliPoly} item 4.
   Consider now $v\in V(G)$ with $\d_{G}(v) = 1$.
   Notice that, $V(G)\setminus \{v\}$ is a defensive $1$-alliance in $G$, so, we have two defensive $1$-alliances in $G$. So, it is a contradiction since $A(G;x)$ is monic polynomial of degree $n+3$ only one term with degree greater than $n+1$. Therefore, $n_1 \leq n + 1$.

   If $n_1 = n+1$, then $1 \leq \d_{n_1} \leq 2$ by \eqref{eq:Wn1}. Besides, if $\d_{n_1} = 1$, then $V(G)$ is an exact defensive $1$-alliance and so, appear a term $x^{n_1 + 1} = x^{n+2}$; this is a contradiction with \eqref{eq:Wn}, thus $\d_{n_1} = 2$.
   Consider now $v\in V(G)$ with $\d_{G}(v) = 2$ and $v_1$ the adjacent vertex of $v$ different to $w$. If $\d_{G}(v_1) \geq 3$, is easy to check that $V(G)\setminus\{v\}$ is a defensive $1$-alliance in $G$, so, since $V(G)$ is an exact defensive $2$-alliance, we have that the number of defensive $1$-alliance in $G$ is at least $2$, but this is a contradiction with \eqref{eq:Wn}. This implies that any vertex in $G$ with degree $2$ is adjacent to $w$ and other vertex of degree $2$.
   Assume that $v_1,v_2 \in V(G)$ with $v_1 \sim v_2$ and $\d_{G}(v_1) = 2 = \d_{G}(v_2)$.
   If $n > 4$, then $V(G)\setminus\{v_1,v_2\}$ is a defensive $1$-alliance in $G$ and so, $A(G;x) \neq A(W_n;x)$ since $G$ have at least $2$ defensive $1$-alliances. If $n=4$, then $W_4 \simeq K_4$ and \cite[Theorem 3.9]{CRST} gives the uniqueness of $A(W_4;x)$. So, such graph $G$ with order $n+1$ such that $A(G;x)=A(W_n,x)$ does not exist.

   Now, we assume $n = n_1$.
   Our next claim is that $\d_2 = 3$.
   Seeking for a contradiction assume that $\d_2>3$. Consider $v_2\in V(G)$ with $\d_{G}(v_2)=\d_2\ge4$. So, $S=\{v_2\}$ is an exact defensive ($-\d_2$)-alliance which provides a term $x^{n-\d_2}$. Since $n-\d_2<n-3$, \eqref{eq:Wn} gives that $n-\d_2$ is odd, \emph{i.e.}, $n-\d_2=2t-1$ for some $t\in \mathbb{N}$.
   Note that every $S\subset V(G)$ with $w\in S$ and $|S|=t$ is an exact defensive ($-\d_2$)-alliance, since $\d_S(w)-\d_{\overline{S}}(w)=t-1-(n-t)=2t-1-n=-\d_2$ and $\d_S(v)-\d_{\overline{S}}(v)\ge1-(\d_2-1)> -\d_2$ for every $v\in S\setminus\{w\}$. Therefore, $A_{-\d_2}(G)\ge {n-1 \choose t-1} + 1 > {n-1 \choose t-1} = A_{-\d_2}(W_n)$ and $A(G;x)\neq A(W_n,x)$. This is the contradiction we were looking for, so $\d_2\le3$.
   Besides, if $\d_2<3$, then $G$ is a proper subgraph of $W_n$ and \cite[Proposition 2.7]{CRST} gives that $A(G;x)\neq A(W_n;x)$. Thus, we have $\d_2=3$.
   Taking $\d_2=3$ in the previous argument, we obtain that $A_{-3}(G)$ is the number of vertex with degree $3$ plus the number of subset $S$ with $\{w\}\subset S\subset V(G)$ and $|S|=(n-2)/2$. Thus, $A_{-3}(G)=A_{-3}(W_n)$ if and only if $\d_2=\d_n$ (\emph{i.e.}, every vertex in $V(G)$ different of $w$ has degree $3$).
   So, we have that degree sequence of both graphs is $\{n-1,3,\ldots,3\}$; in fact, $G_C:=\langle V(G)\setminus\{w\}\rangle$ is a $2$-regular graph.
   Finally, seeking for a contradiction assume that $\langle V(G)\setminus\{w\}\rangle$ is disconnected. Note that $G$ is an isomorphic graph to $E_1+G_C$, besides, $G_C$ is a disjoint union of $r\ge2$ cycle graphs. Without loss of generality we can assume that $G_C \simeq \cup_{i=1}^{r} C_{n_i}$ with $\sum_{i=1}^{r} n_i = n-1$. Hence, it suffices to prove $A(G_C;1)\ne A(C_{n-1};1)$ by Theorem \ref{t:join}.
   Then, by Proposition \ref{p:Disjoint_Cn} we have $A(G_C;1)=A(C_{n_1};1)+\ldots + A(C_{n_r};1)=\sum_{i=1}^{r} \left({n_i}^2 - n_i + 1\right)$. It is easy to check that $\sum_{i=1}^{r} \left({n_i}^2 - n_i + 1\right) < \left(\sum_{i=1}^{r} n_i \right)^2 - \left(\sum_{i=1}^{r} n_i\right) + 1$. This is the contradiction we were looking for, therefore, $G_C$ is connected.
\end{proof}

\section{Unimodality of the alliance polynomial of wheel graphs}\label{unimodal}
Unimodal polynomials arise often in combinatorics, geometry and algebra. The reader is referred to \cite{B,S} for surveys of the diverse techniques employed to prove that specific families of polynomials are unimodal. In this section we show that the alliance polynomial of several wheel graphs are unimodal.

A finite sequence of real numbers $(a_{0} , a_{1} , a_{2} , . . . , a_{n})$ is said to be \emph{unimodal} if there is some $k \in \{0, 1, . . . , n\}$, called the \emph{mode} of the sequence, such that
  $$a_{0} \leq . . . \leq a_{k-1} \leq a_{k} \quad \text{and} \quad a_{k} \geq a_{k+1} \geq . . . \geq a_{n};$$
  the mode is unique if $a_{k-1} < a_{k}$ and $a_{k} > a_{k+1}$.
A polynomial is called unimodal if the sequence of its coefficients is.
Since several times there will be a lot of coefficients in an alliance polynomial that are zeros, along this paper we say that an alliance polynomial with coefficient zero in every even (odd, respectively) powers of $x$ is unimodal if the sequence of its non-zero coefficients is unimodal.

In \cite{CRST,CRST3}, authors study unimodality of alliance polynomial of some graphs: path, cycle, complete, complete without one edge, and cubic graphs. No all of them are unimodal, for instance, the alliance polynomial of path graphs $A(P_n;x)$ and complete graph without one edge $A(K_n/e;x)$ are unimodal if and only if $2\le n\le 4$.
We show that the alliance polynomial of wheel graphs with even order $A(W_{2n};x)$ are unimodal. 
The following result is a direct consequence of Propositions \ref{p:Wn} and \ref{p:Disjoint_Cn}

\begin{theorem}\label{t:W2n}
  $A(W_{2n};x)$ is unimodal with mode $A_{-1}(W_{2n})$.
\end{theorem}

\begin{proof}
By Proposition \ref{p:Wn}, we have
 \[
 \begin{aligned}
   A(W_{2n};x) = &\quad x + (2n-1)x^{2n-3} + (2n-1)(2n-3)x^{2n-1} + x^{2n+1} + \\
             & + \displaystyle\sum_{k=2}^{n-1} {2n-1 \choose k-1} x^{2k-1} + {2n-1 \choose n-1} x^{2n-1} + \\
             & + \displaystyle\sum_{k=n+1}^{2n-1} a_{2n,k} x^{2n-1} + \sum_{k=n+1}^{2n-1} b_{2n,k} x^{2n+1}   + x^{2n+3}.
 \end{aligned}
 \]

First note that $A_{-2n+1}(W_{2n})< A_{-2n+3}(W_{2n})< \ldots < A_{-1}(W_{2n})$. Besides, $A_{1}(W_{2n})> A_{3}(W_{2n})$.

Now we claim $a_{m,r-1}\geq b_{m,r}$ for every $2\leq r\leq m-1$. Let $r$ be fixed. 
From each of the $b_{m,r}$ configurations we can obtain a configuration out of the $a_{m,r-1}$ configurations by replacing the second digit $1$ in the first block of ones\footnote{We consider the first block of 1s in the configuration to the block of ones which starts with the minimum index of vertices.} with a $0$. Since those $b_{m,r}$ configurations are different, we obtain $b_{m,r}\leq a_{m,r-1}$. 

On the other hand, we have that $b_{2n,n}<{2n-1 \choose n}={2n-1 \choose n-1}$. Therefore, we obtain $A_{-1}(W_{2n}) > A_{1}(W_{2n})$. 
\end{proof}


\begin{thebibliography}{1}
\bibitem{ALL}	J. L. Arocha and B. Llano, Mean value for the matching and dominating polynomial, {\it Discuss. Math. Graph Theory} {\bf 20}(1), (2000) 57-69. 

\bibitem{BCLS}  L. A. Basilio, W. Carballosa, J. Lea\~nos and J.M. Sigarreta, On the differential polynomial of a graph, {\it Acta Mathematica Sinica, English Series} {\bf 35}(3) (2019), 338-354.

\bibitem{B}  F. Brenti, Log-concave and unimodal sequences in Algebra, Combinatorics and Geometry: an update. Contemporary Mathematics, 178, 71-84, 1994.



\bibitem{CRST} W. Carballosa, J.M. Rodr\'{\i}guez, J. M. Sigarreta and Y. Torres-Nu\~nez, Computing the alliance polynomial of a graph, {\it Ars Combinatorics} {\bf 135} (2017), 163-185. 

\bibitem{CRST3} W. Carballosa, J.M. Rodr\'{\i}guez, J. M. Sigarreta and Y. Torres-Nu\~nez, Alliance polynomial of regular graphs, {\it Discrete Applied Mathematic} {\bf 225} (2017), 22-32.

\bibitem{F} E.J. Farrell, An introduction to matching polynomials, {\it Journal of Combinatorial Theory Serie B} {\bf 27} (1979) 75-86.






\bibitem{HL} C. Hoede and X. Li, Clique polynomials and independent set polynomials of graphs, {\it Discrete Mathematics} {\bf 125} (1994) 219-228.

\bibitem{MN} A. de Mier and M. Noy, On graphs determined by their Tutte polynomials, {\it Graphs Combin.} {\bf 20}(1) (2004), 105-119.

\bibitem{N} M. Noy, On graphs determined by polynomial invariants, {\it Theoretical Comp. Sci.} {\bf 307}(2) (2003), 365-384.





\bibitem{S} R. Stanley, Log-concave and unimodal sequences in algebra, combinatorics and geometry. {\it Graph theory and its applications: East and West} (Jinan, 1986), 500-535, Ann. New York Acad. Sci., 576, New York, 1989.


\bibitem{T} W. T. Tutte, A contribution to the theory of chromatic polynomials, {\it Canadian Journal of Mathematics} (1954) {\bf 6}, no 80-91, p. 3-4.
\end{thebibliography}
\end{document}